\documentclass[11pt]{article}
\usepackage{amsfonts}
\usepackage{amssymb}
\usepackage{amscd}
\usepackage{amsmath}
\usepackage{epsfig}
\usepackage{graphicx, subfigure}
\usepackage{latexsym}
\usepackage{pst-3dplot}
\usepackage[symbol*]{footmisc}
\usepackage{pdfsync}
\usepackage{hyperref}
\usepackage{multirow}
\usepackage[utf8]{inputenc}
\usepackage{pdfsync}
\usepackage{float}

\usepackage{color}
\usepackage{xcolor}
\usepackage{epsfig}
\usepackage{graphicx, subfigure}
\usepackage[makeroom]{cancel}
\usepackage{ulem}

\newlength{\dinwidth}
\newlength{\dinmargin}
\usepackage[all,cmtip]{xy}

\newtheorem{definition}{Definition}
\newtheorem{theorem}{Theorem}

\newtheorem{proposition}{Proposition}
\newtheorem{corollary}{Corollary}
\newtheorem{remark}{Remark}
\newtheorem{lemma}{Lemma}
\newtheorem{example}{Example}

\def\Z{\mathcal Z}
\def\Q{\mathcal Q}
\def\F{\mathbb F_2}

\def \P{\mathcal P}

\def \a{\alpha}
\def \b{\beta}

\def\D{\mathcal D}
\def\d{\mathcal D}
\def\B{\mathcal B}

\begin{document}

\renewcommand*{\thefootnote}{\fnsymbol{footnote}}

\title{Block designs and  systems of pairs}
\author{Alexander Shramchenko and Vasilisa Shramchenko$^*$}

\date{}

\maketitle

\footnotetext[1]{Department of mathematics, University of
Sherbrooke, 2500, boul. de l'Universit\'e,  J1K 2R1 Sherbrooke, Quebec, Canada. E-mail: {\tt Vasilisa.Shramchenko@Usherbrooke.ca}}

 \begin{abstract}
We propose a new way to construct one block design from another, thus linking some block designs in a new way. At the same time, we find relationships between block designs and systems of pairs.  In particular, these relationships  allow us to generate some systems of pairs. 
 
  \end{abstract}
%
  

\section{Introduction}

A combinatorial object known as {\it a system of pairs} has been largely forgotten. However, it has connections to various topics in mathematics. It appeared as early as 1859 in \cite{Reiss}.   In \cite{Anderson_topology} systems of pairs were used in the context of enumeration of mutually complementary topologies on a given finite set. System of pairs of a certain type are studied in \cite{Korovina}. In \cite{Rokowska}, they appear in a context of resolvable block designs, and in \cite{Hanani}, systems of pairs are used in a study of quadruple systems. 

 A prototype for systems of pairs is the special case where each {\it subsystem} contains $t=v/2$ pairs, where $v$ is the total (even) number of elements. In this special case, each subsystem is a perfect matching of a complete graph with $v$ vertices. Thus the systems of pairs take origin in the theory of 1-factorizations of complete graphs (partitions of the set of edges of the graph into perfect matchings) dating back to \cite{Konig} first published in 1936, see also \cite{Anderson_Room}. General systems of pairs correspond to decompositions of the edge set of a complete graph into matchings of prescribed size.
Graph decpositions is a currently active area of research focused on understanding their structure, existence, algorithms, and applications, see \cite{Csaba, Fomin,  Keevash, Ma} and references therein.

Even though edge set decompositions fit naturally into design theory, see for example \cite{encyclopedia} or VI.24 of \cite{handbook}, 
the connection between system of pairs and block designs has not yet been sufficiently explored. The purpose of this paper is to bring back this question by showing that some systems of pairs can be constructed using block designs and vice versa. 

More precisely, we show (Theorem \ref{thm_main}) that starting with   any symmetric design or a quasi-symmetric design satisfying a particular property built on a set $V$, one may construct another block design. This new block design  may contain identical blocks even if the original one does not. The blocks of the new block design are indexed by unordered pairs of elements of $V$. The pairs corresponding to identical blocks are then grouped together to form {\it parallel classes}.  The collection of these parallel classes satisfies a symmetry property (Theorem \ref{thm_symmetric}) and may form a system of pairs (Theorem \ref{thm_iff} and Proposition \ref{prop_auto}). This process may sometimes be iterated (Example \ref{example_designpairs}).  On the other hand, we show that starting with a system of pairs of certain type, one may construct block designs (Theorem \ref{thm_pairstoblocks}).

\section{From one block design to another}
\label{sect_new}

A balanced incomplete block design is defined as follows, see \cite{encyclopedia, handbook, Rib}.

\begin{definition}
\label{def}
Let $V$ be a finite set of cardinality $v=|V|$ and let $b,k,r,\lambda$ be four positive integers with $k<v$. A balanced incomplete block design (BIBD) with parameters $(v,b,r,k,\lambda)$ built on the set $V$ is a collection of $b$ blocks, each of which is a $k$-element subset of $V$, such that every element of $V$ is contained in exactly $r$ blocks and every pair of elements of $V$ is contained in exactly $\lambda$ blocks.
\end{definition}

We will say simply {\it block design} or {\it design} to mean BIBD. The elements of the set $V$ are also referred to as {\it treatments} or {\it varieties}. 
There may be several block designs or none for given values  of $(v,b,r,k,\lambda)$.  Definition \ref{def} implies the following necessary conditions for the existence of a BIBD: $bk=vr$ and $r(k-1)=\lambda(v-1)$. Parameters satisfying these conditions are called {\it admissible}. Another necessary condition for a design with $ 2\leqslant k<v$ to exist is Fisher's inequality, $b\geqslant v.$

In general, some blocks in a design may coincide. Designs without repeated blocks are called {\it simple}. A block design is called {\it symmetric} if $v=b$ (or if $r=k$).   A symmetric design with $\lambda=1$ is called a {\it projective plane}.

We denote by $\D$ a block design from Definition \ref{def} built on the set $V=\{1, 2, \dots, v\}$ with parameters  $(v,b,r, k,\lambda)$. The blocks of $\D$ are assumed ordered in some arbitrary way from $1$ to $b$ and denoted by $B_n$ with  $n=1, \dots,  b.$

 The {\it incidence matrix} $M=M(\D)$ of $\D$ is a matrix of the size $v\times b$ 
such that the entry $M_{ij}$ is 1 if $i\in B_{j}$ and 0 otherwise: $M_{ij}= |\{i\}\cap B_n| = \chi_{B_j}(i),$ where $\chi_{B_j}(t)$ is a characteristic function of the set ${B_j}$. This matrix satisfies
\begin{equation}
\label{MMt}
MM^T=(r-\lambda)I+\lambda J,
\end{equation}
where $I=I_v$ is the identity matrix of size $v$ and $J$ is a $v\times v$ matrix whose each entry is 1. 

For a design $\D$ on a set $V$, for any pair $\{x,y\}\subset V$, define $\Z(x,y)$ as a $b\times1$ matrix such that its $j$th entry is
\begin{equation}
\label{Z}
\Z_j(x,y)= \chi_{B_j}(x)-\chi_{B_j}(y). 
\end{equation}
In other words, if $(e_1, \dots, e_v)$ is the canonical basis of $\mathbb R^v$ where $e_i$ is a column vector such that $(e_i)_j=\delta_{ij}$, we have
\begin{equation}
\label{ZM}
\Z(x,y)= (e_x-e_y)^TM.
\end{equation}
It is shown in \cite{S2}, that vectors \eqref{ZM} span the eigenspace of dimension $v-1$ of the matrix $M^TM$ corresponding to the eigenvalue $r-\lambda$. The remaining eigenspace of $MM^T$ is spanned by the vector $(1,\dots, 1)$ corresponding to the eigenvalue $rk.$

Let $\Z^+(x,y)$ stand for the column vector  whose $j$th component is 
\begin{equation}
\label{zplus}
\Z_j^+(x,y):=|\Z_j(x,y)|. 
\end{equation}
Note that $\Z^+(x,y)=\Z^+(y,x)$ and thus this vector corresponds to an unordered pair $\{x,y\}$. 

Denote by $\Z^+$ the matrix of size $b\times \frac{v(v-1)}{2}$ composed of the column vectors $\Z^+(x,y)$ built from $\D$ by \eqref{ZM}, \eqref{zplus} where $(x,y)$ runs through the set of unordered pairs of elements from $V$. An entry of $\Z^+$ is either 0 or 1. 
Let us write $\Z^+_j$ for the $j$th row of $\Z^+$, the row corresponding to the block $B_j.$  We will also write $\Z^+(\D)$ for the matrix $\Z^+$ built from a design $\D$ when several designs are involved.
\\

\begin{remark}
\label{rmk_complement}
Consider the block design $\bar \D$ whose blocks are $\bar B_i=V\setminus B_i$, the complements of the blocks of $\D$. 
Then $\Z^+(\D)=\Z^+(\bar \D).$ This follows readily from definition \eqref{zplus} of the matrix $\Z^+.$
\end{remark}

Let us also note that in the case $k\neq \frac{v}{2}$, the design $\D$ and its complement $\bar \D$ can be reconstructed from the matrix $\Z^+=\Z^+(\D). $ Indeed, let $V=\{1, \dots, v\}$ and assume the matrix $\Z^+$ to be given. Let us fix $1\leq j\leq b$ and suppose  that $1\in B_j.$ This assumption and the $j$th components of the column vectors $\Z^+(1,2), \dots, \Z^+(1,v)$ allow us to reconstruct $B_j$ completely: for every $n\in V$, we have $n\in B_j$ if and only if $\Z^+(1,n)=0$. Similarly, the assumption $1\notin B_j$ allows us to reconstruct $B_j$ from $\Z^+$: for every $n\in V$, we have $n\in B_j$ if and only if $\Z^+(1,n)=1$. If the block sizes of $\D$ and $\bar \D$ are different, considering both assumptions, it is possible to place the reconstructed blocks in $\D$ or $\bar \D$ in an unambiguous way.

\begin{lemma}
\label{lemma-main}
Let $\D$ be a design with parameters $(v,b,r,k,\lambda)$. For $i\neq j$, denote $\mu_{ij}=|B_i\cap B_j|$ the number of common elements in the blocks $B_i$ and $B_j$. Then for the rows $\Z^+_i $ and $\Z^+_j$ of $\Z^+=\Z^+(\D)$ corresponding to the blocks $B_i$ and $B_j$, we have 
\begin{equation}
\label{newlambda}
\Z^+_i (\Z^+_j)^T =\mu_{ij}(v-2k+\mu_{ij})+(k-\mu_{ij})^2.
\end{equation}
\end{lemma}
The lemma thus states that, for given $i\neq j$ there are exactly $\mu_{ij}(v-2k+\mu_{ij})+(k-\mu_{ij})^2$ unordered pairs $\{x,y\}$ of elements of $V$ for which $\Z^+_i(x,y)=\Z^+_j(x,y)=1$.
\\

{\it Proof.} Consider two sets of $k$ elements, $B_i$ and $B_j$ with $i\neq j$, having $\mu_{ij}$ elements in common. We want to count the number of pairs of elements $\{x,y\}$ for which exactly one element belongs to $B_i$ and exactly one element belongs to $B_j$. There are two ways to form such a pair. One way is to take one element from $B_i\cap B_j$ and another one from $V\setminus(B_i\cup B_j)$. This gives $\mu_{ij}(v-2k+\mu_{ij})$ distinct pairs. Another way is to take one element from $B_i\setminus(B_i\cap B_j)$ and another one from $B_j\setminus(B_i\cap B_j)$. This gives $(k-\mu_{ij})^2$ distinct pairs. 
\hfill $\Box$

  A block design is symmetric if and only if every two distinct blocks have $\lambda$ elements in common, see II.6.1 of \cite{handbook}, that is the {\it intersection numbers} $\mu_{ij}=\lambda$ for any $i,j.$ As a natural generalization of symmetric designs, the quasi-symmetric designs are designs with exactly
two intersection numbers, $\mu_{ij}\in\{\mu_1, \mu_2\}$ for all $1\leqslant i,j \leqslant b$ with $\mu_1\neq \mu_2$, see \cite{quasi}. Quasi-symmetric designs are known to have deep connections to strongly regular graphs via their block graphs.   The next theorem gives a way to construct a new block design starting with a symmetric design or a subclass of quasi-symmetric ones. 
\begin{theorem}
\label{thm_main}
For a block design $\D$ with parameters $(v,b,r,k,\lambda)$, the matrix $\Z^+=\Z^+(\D)$ is an incidence matrix of another block design $\D'$ if and only if $\mu_{ij}=|B_i\cap B_j|$ as a function of $i,j\in\{1, \dots, b\}$ for $i\neq j$ takes at most two values $\mu_1$ and $\mu_2$ for which 
\begin{equation}
\mu_1=\mu_2\qquad\text{ or } \qquad \mu_1+\mu_2=2k-\frac{v}{2}.
\label{quadratic}
\end{equation}
In this case, the block design $\D'$ has parameters $(v',b',r',k',\lambda')$ where 
\begin{equation}
\label{parameters}
v'=b, \qquad b'=\frac{v(v-1)}{2}, \qquad r'=k(v-k), \qquad k'=2(r-\lambda), \qquad \lambda' = 2\mu_1^2+\mu_1(v-4k)+k^2. 
\end{equation}
Alternatively, $\lambda'$ can be obtained by 
\begin{equation}
\label{lambda-prime}
 \lambda' = \frac{k(v-k)(2r-2\lambda-1)}{b-1}. 
\end{equation}

\end{theorem}

{\it Proof.}  Suppose, there exists a design $\D'$ for which $\Z^+$ is an incidence matrix. The matrix $\Z^+$ has dimensions $b\times \frac{v(v-1)}{2}$. Therefore, $\D'$ is built on a set $V'$ of $b$ elements and has ${v(v-1)}/{2}$ blocks. To find $r'$, we need to find the number  of nonzero elements in each row of $\Z^+.$ This is the number of pairs of elements $\{x,y\}\subset V$ for which only one element belongs to a given block of $\D$, that is $k(v-k)$. To find $k'$, we need to find the number of nonzero elements in each column of $\Z^+$. This is the number of blocks of $\D$ containing exactly one element of a given pair $\{x,y\}$. Since each of $x$ and $y$ appears in $r$ blocks of $\D$ and there are $\lambda$ blocks in which they both appear, we obtain that $k'=2(r-\lambda)$.

Given that we obtain constant values for $v', b', r'$ and $k'$, starting with any design $\D$, the only condition for the existence of a design $\D'$ is the existence of a value $\lambda'$ independent of the choice of a pair of elements of the set $V'=\{1, \dots, b\}$. In other words, we need a design $\D$ such that for any of its distinct blocks $B_i$ and $B_j$, the number of  columns of $\Z^+$ containing 1 in both lines $i$ and $j$ is the same for all $i$ and $j$ with $i\neq j$. This number is given by \eqref{newlambda} and does depend on $i$ and $j$ in general. 

Thus we can say that if and only if a design $\D$ is such that the value \eqref{newlambda} is the same for all $i\neq j$, then $\Z^+$ is an incidence matrix of some block design $\D'$ with $\lambda'$ given by \eqref{newlambda}. Since \eqref{newlambda} is quadratic in $\mu_{ij}$, we can allow for a maximum of two values $\mu_1$ and $\mu_2$ of $\mu_{ij}$. 
 Equating expression \eqref{newlambda}  for both values of  $\mu_{ij}$, we obtain the equation $2\mu_1^2+\mu_1(v-4k)=2\mu_2^2+\mu_2(v-4k)$. This equation factorizes to give the two possibilities in \eqref{quadratic}. 
The quadratic expression for $\lambda'$ in \eqref{parameters} coincides with that in \eqref{newlambda} for $\mu_{ij}=\mu_1.$ 

Having proved that $\D'$ is a block design, we can obtain $\lambda'$ from the relation $r'(k'-1) = \lambda'(v'-1)$ valid for parameters of a block design. This gives formula \eqref{lambda-prime}. 
\hfill $\Box$

\begin{remark} Denote by $\B$ the set of blocks of a design $\D$ built on a set $V$. Then it is customary to write $\D=(V, \B)$. If a design $\D'$ from Theorem \ref{thm_main} exists, then $\D'=(\B, V^{(2)})$, where $V^{(2)}$ stands for the set of unordered pairs of elements of $V$. 
\end{remark}

\begin{corollary}
\label{cor_symmetric}
If $\D$ is a symmetric design with parameters $(v=b,r=k,\lambda)$, then there exists a design $\D'$ from Theorem \ref{thm_main}. 
\end{corollary}

{\it Proof.}  
In a symmetric design every two distinct blocks have $\lambda$ elements in common, II.6.1 of \cite{handbook}. 
Thus the condition of Theorem \ref{thm_main} for the intersection numbers is satisfied: $\mu_{ij}=\lambda=\mu_1=\mu_2$. From \eqref{parameters}, we obtain $\lambda' = 2\lambda^2+\lambda(v-4k)+k^2.$
\hfill $\Box$

\begin{corollary}
A design $\D$ gives rise to a design $\D'$ from Theorem \ref{thm_main} if and only if its complement $\bar \D$ gives rise to a design $\bar\D'$ as described in Theorem \ref{thm_main}. In this case  the two designs $\D'$ and $\bar\D'$ coincide. 
\end{corollary}

{\it Proof. This follows from Remark \ref{rmk_complement} as both designs $\D$ and $\bar\D$ correspond to the same matrix $\Z^+$.}

\section{Systems of pairs}
\label{sect_systems}
In this section, following \cite{Korovina, Rokowska} we define systems of pairs, their particular types and discuss some relationships with block designs. 
\begin{definition}
\label{def_pairs}
Given a set $V=\{1, 2, \dots, v\}$, a system of pairs $\P$ with parameters $(v,  n, t)$ on $V$ is an arrangement of all $\frac{v(v-1)}{2}$ unordered pairs into $n$ families $\P_1, \dots, \P_n$ (called subsystems or parallel classes) of $t$ pairs each for some integer $n,t\geqslant 1$ satisfying $nt=\frac{v(v-1)}{2}$ in such a way that no subsystem contains any element of $V$ more than once.  A system of pairs with $t=1$  is called trivial. 
\end{definition}
In the references, this definition is often used for even $v$ with $n=v-1$ and $t=v/2$.  In this case, the system of pairs gives a 1-factorization of  a complete graph with $v$ vertices. The next two definitions introduce two notable classes of systems of pairs. 
\begin{definition}
\label{def_symmetric-pairs}
A system of pairs $\P$ is called symmetric if for any two distinct unordered pairs $\{x,y\}$ and $\{\a,\b\}$ belonging to a subsystem $\P_l$ of $\P$ there exist subsystems $\P_i$ and $\P_j$ of $\P$ such that $\P_i$ contains the unordered pairs $\{x,\a\}$ and $\{y,\b\}$ while $\P_j$ contains $\{x,\b\}$ and $\{y,\a\}.$
\end{definition}
Let us assume the pairs in each subsystem are ordered in some way from 1 to $t$ and write $\P_j(i)$ for the pair number $i$ in the subsystem $\P_j.$
\begin{definition}
\label{def_simple-pairs}
A nontrivial symmetric system of pairs $\P$ with parameters $(v, n, t)$ on the set $V$ is called simple if for any integer $2<q\leqslant t$  and $q$-subsets $s_l=\{i_{l1}, \dots, i_{lq}\}$ of the set of indices $\{1, \dots, t\}$, we have  $f_{j_1}(s_1) = f_{j_2}(s_2)$ if and only if $j_1= j_2$ and $s_1=s_2$ for the following subsets $f_j(s_l)$  of $V$:
\begin{equation}
\label{f}
f_j(s_l) = \underset{m\in s_l}{\cup}\P_j(m), \quad j=1, \dots, n, \quad l=1, \dots, {t\choose q}.
\end{equation}
\end{definition}

The systems of pairs from Examples \ref{example_Fano},  \ref{example_designpairs2} and \ref{example_designpairs} below are simple.

\begin{theorem}
\label{thm_pairstoblocks}
Let $\P$ be a nontrivial symmetric system of pairs with parameters $(v,  n, t)$ built on the set $V=\{1, 2, \dots, v\}$. 
\begin{enumerate}
\item \label{1} Let  $x,y\in V$ be some fixed elements. There are precisely $2t-1$ subsystems of $\P$ each of which contains both $x$ and $y$ as elements of some pairs. 

\item
\label{2}
For $v\neq 4$, there exists a simple block design $\D_2(\P)$ with parameters 
\begin{equation}
\label{D2param}
v_2=v, \qquad b_2=\frac{n}{3}{t\choose 2}, \qquad r_2=\frac{1}{3}(v-1)(t-1), \qquad k_2=4, \qquad \lambda_2=t-1.
\end{equation}
\item
\label{3}
 If $\P$  is simple, then for any $2<q\leqslant t,\; v\neq 2q$ there exists a  simple  block design $\D_q(\P)$ with parameters 
\begin{equation*}
v_q=v, \qquad b_q=n{t\choose q},\qquad r_q=(v-1){{t-1}\choose {q-1}}, \qquad k_q=2q, \qquad \lambda_q = (2q-1){t-1\choose q-1}.
\end{equation*}
\end{enumerate}
\end{theorem}
{\it Proof.}  The restrictions on $v$ in items \ref{2} and \ref{3} are imposed to avoid the case when $v_2=k_2$ or $v_q=k_q$ for the resulting design as, according to Definition \ref{def}, we do not consider designs with $k=v$ in this paper. If we decide to allow designs with $v=k$, then the statements and proofs of this theorem hold without restriction.  
\begin{enumerate}
\item The pair $\{x,y\}$ belongs to only one subsystem of $\P$, let us denote it by $\P_1.$ If $t>1$, there are pairs $\{x_2,y_2\}, \dots, \{x_t,y_t\}$ in $\P_1$ having no elements in common pairwise and not intersecting $\{x,y\}$. By symmetry of $\P$, for any $j=2, \dots, t$, the pairs $\{x,x_j\}$, $\{y, y_j\}$ belong to some subsystem that we denote $\P_j$ and the pairs $\{x,y_j\}$, $\{y, x_j\}$ belong to a subsystem that we denote $\P'_j$ with $\P_j\neq \P'_j.$ Thus there are at least $1+2(t-1)$ subsystems $\P_1, \P_j, \P'_j$ each containing $x$ and $y$ in their pairs.

Let us prove that there are no other subsystems containing both $x$ and $y$. Suppose there is a subsystem $\P_0$ containing two disjoint pairs $\{x,\a\}$ and $\{y,\b\}$. Due to symmetry of $\P$, there exists another subsystem containing pairs $\{x,y\}$, $\{\a,\b\}$, and by unicity of subsystem containing $\{x,y\}$,  this subsystem is $\P_1.$ Therefore there is an index $j\in\{2, \dots, t\}$ such that $\{\a,\b\}=\{x_j,y_j\}$ and thus $\P_0$ coincides with either $\P_j$ of $\P'_j$

\item Let us construct the design $\D_2(\P)$ from $\P$ as follows. In every subsystem $\P_j$ for $j=1, \dots, n,$ consider all possible unions of two distinct pairs from $\P_j$. These unions become blocks of $\D_2(\P)$. More precisely, 
for $\P_j=\{\{x_{j1}, y_{j1}\}, \dots, \{x_{jt}, y_{jt}\}\}$, define $B^j_{pq}=\{x_{jp}, y_{jp}, x_{jq}, y_{jq}\}$ for all $q< p$ with $1\leqslant p, q\leqslant t$. Given that no two pairs from the same subsystem have common elements, this is a well-defined procedure resulting in $t\choose 2$ distinct four-element sets for every $j$. Due to the symmetry of $\P$, the set of $B^j_{pq}$ for $j=1, \dots, n$ contains three copies of every four-element set. Removing repetitions, we obtain a set $\B$ of $\frac{n}{3}{t\choose 2}$ distinct blocks. 
Let us prove that  the set $\B$ is a set of blocks of a simple block design $\D_2(\P)$.

By definition, there is only one subsystem $\P_{j^*}$ in a system of pairs containing a given pair $\{x,y\}$. Therefore, the blocks containing $\{x,y\}$ are the  $B^{j^*}_{pq}$  obtained by unions of $\{x,y\}$ with all the remaining pairs of $\P_{j^*}$. This gives $t-1$ distinct blocks. Consider now another subsystem $\P_{j'}$ containing pairs $\{x,\a\}$ and $\{y,\b\}$. There is thus a four-element set $B^{j'}_{pq}=\{x,y,\a,\b\}$ containing the pair $\{x,y\}$. By symmetry of $\P$, there exists a subsystem containing the pairs $\{x,y\}$ and $\{\a,\b\}$. This subsystem is $\P_{j^*}$ and thus $B^{j'}_{pq}=\{x,y,\a,\b\}$ is already included in the set of $t-1$ distinct blocks obtained from the subsystem $\P_{j^*}$. This proves that every pair $\{x,y\}$ appears $\lambda:=t-1$ times in the elements of $\B$. 

Let us now calculate the number of occurrences of each element $x\in V$ in the blocks $B^j_{pq}$. The number of subsystems of $\P$ containing a pair $\{x,y\}$ for some  $y\in V$ is $v-1$. Consider one such subsystem $\P_{j^*}$. There are $t-1$ pairs in $\P_{j^*}$ which can be merged with the pair $\{x,y\}$ to form a four-element set $B^{j^*}_{pq}$. Thus, altogether, we obtain $(v-1)(t-1)$ four-element sets before identifying the identical ones. After this identification, given that every four-element set is obtained three times by merging pairs as described, we have that every element appears $r:=\frac{1}{3}(v-1)(t-1)$ times in the elements of $\B.$

\item The blocks of the design $\D_q(\P)$ are the sets of $2q$ elements $f_j(s_l)$  defined by \eqref{f}. Since $\P$ is simple, all these subsets of $V$ are distinct. Thus $b=n{t\choose q}.$ Let us count the number of occurrences of each  $x\in V$ in these blocks. A given element $x$ belongs to $v-1$ pairs. These pairs belong to $v-1$ blocks, one pair per block. Therefore $x$ appears in $r:=(v-1){t-1\choose q-1}$ blocks. 

Let us now consider an arbitrary pair $\{x,y\} \in V$. This pair belongs to one subsystem of $\P$, let us denote it by $\P_{j_*}$. This subsystem gives rise to $t-1\choose q-1$ blocks. Due to item \ref{1} of the present theorem, there remains $2(t-1)$ subsystems of $\P$ different from $\P_{j_*}$ that contain $x$ and $y$ in their pairs. Each of these $2(t-1)$ subsystems gives rise to $t-1 \choose q-2$ blocks containing $\{x,y\}$. Thus, the pair $\{x,y\}$ belongs to 
$\lambda:= {t-1\choose q-1}+2(t-1){t-2\choose q-2}$ blocks, which gives the required expression for $\lambda$. 
\end{enumerate}
\vspace{-0.4cm}
\hfill $\Box$

Note that one can reconstruct $\P$ from $\D_2(\P)$ defined in the proof of item \ref{2} of Theorem \ref{thm_pairstoblocks} by grouping each pair $\{x,y\}$ with all pairs $\{\a,\b\}$ such that $\{x,y,\a,\b\}$ is a block of $\D_2(\P)$.

\section{A system of pairs corresponding to a block design}
\label{sect_design-pairs}
In this section, we show that one can associate a system of pairs with some block designs. Note that the design $\d'$ obtained from $\D$ by the procedure of Theorem \ref{thm_main} may contain repeated blocks even if $\D$ is simple. The groups of identical blocks correspond to equivalence classes of pairs as defined by the following definition. 

\begin{definition}
\label{def_equivalence}
Given a block design $\D$, we say that two unordered pairs $\{x,y\}, \{\a,\b\}\subset V$ are equivalent to each other (with respect to $\D$) and we write $\{x,y\} \sim_\D \{\a,\b\}$, if  $\Z^+(x,y)=\Z^+(\a,\b),$ where $\Z^+=\Z^+(\D)$ is defined  by \eqref{zplus} for the design $\D$. 
\end{definition}
The equivalence relation from Definition \ref{def_equivalence} partitions the set $V^{(2)}$ of unordered pairs of elements of $V$ into $n$ equivalence classes for some natural number $n$.  Let us denote these equivalence classes by $\P_l(\D)$ and the set of the classes by $\P(\D)=\{\P_1(\D), \dots, \P_n(\D)\}$. 
The next definition extends  Definition \ref{def_symmetric-pairs} of symmetric systems of pairs to $\P(\D)$. 
\begin{definition}
\label{def_symmetry_classes}
Given a design $\D$, the set of equivalence classes $\P(\D)$ is called symmetric if for any two distinct unordered pairs $\{x,y\}$ and $\{\a,\b\}$ belonging to a class $\P_l(\D)$ there exist classes $\P_i(\D)$ and $\P_j(\D)$ such that $\P_i(\D)$ contains the unordered pairs $\{x,\a\}$ and $\{y,\b\}$ while $\P_j(\D)$ contains $\{x,\b\}$ and $\{y,\a\}.$
\end{definition}
\begin{lemma}
\label{lemma_t}
Let $\D$ be a design with parameters $(v, b,r, k, \lambda)$. 
Let $t_j$ denote the number of elements (pairs) in the equivalence class $\P_j(\D)$. Then $1\leq t_j\leq k$.
\end{lemma}

{\it Proof.} 
The class $\P_j(\D)$ contains $t_j$ disjoint pairs $\{x_\alpha, y_\alpha\}$ with $\alpha=1, \dots, t_j.$ There exists a block $B_i$ of $\D$ such that $\Z_i^+(x_\alpha, y_\alpha)=1$ for all $\alpha$. Thus $B_i$ contains exactly one element from each pair $\{x_\alpha, y_\alpha\}$ with $\alpha=1, \dots, t_j.$ Given that these pairs are disjoint and that $B_i$ contains exactly $k$ elements, we obtain that $t_j\leq k$.
\hfill $\Box$
\\

Let us characterize $\Z^+(x,y)$  alternatively as $\Z(x,y)$ from \eqref{Z} considered modulo two, that is viewed over the field $\mathbb F_2=\{0,1\}$.  Let $m_x\in \mathbb F_2^b$ be the (transposed) $x$-th row of $M$ read modulo 2, that is $(m_x)_j=\chi_{B_j}(x).$ Then 
\begin{equation}
\label{mod2}
\Z^+(x,y)=m_x+m_y \;\text{ in } \;\; \F^b.
\end{equation}
\begin{lemma}
\label{lemma_presystem}
Let $\D$ be a design with parameters $(v, b,r, k, \lambda)$ and $\{x,y\}\ne \{\a,\b\}$ be two unordered pairs of $V$. 
\begin{enumerate}
\item \label{presystem1}
$\{x,y\}\sim_\D\{\a,\b\}$ if and only if $m_x+m_y= m_\a+m_\b$ in $\F^b$;

\item 
\label{presystem2}
If $\{x,y\}\sim_\D\{\a,\b\}$, then no two elements of the set $\{\a,\b,x,y\}$ coincide;

\item
\label{presystem3}
 $\{x,y\}\sim_\D\{\a,\b\}$ if and only if $|\{x,y,\a,\b\}|=4$ and $|\{x,y,\a,\b\}\cap B_j|\in\{0,2,4\}$ for every $1\leqslant j\leqslant b.$

\end{enumerate}
\end{lemma}

{\it Proof.} 
\begin{enumerate}
\item This follows from \eqref{mod2} and Definition \eqref{def_equivalence}.

\item Let us suppose that the two pairs meet, for example $x=\a$. Then, by item \ref{presystem1}, we have that $m_y=m_\b$.  In this case, if $y\neq \b$, then these two elements do not appear in any block of $\D$ separately from one another. This implies that $r=\lambda$, which is impossible due to the relation $r(k-1)=\lambda(v-1)$ and the condition $k<v$.  We thus conclude that $y=\b$, which proves that $\{x,y\}\cap \{\a,\b\}\neq \emptyset$ implies $\{x,y\}=\{\a,\b\}$.

\item  If $\{x,y\}\sim_\D\{\a,\b\}$, then by item \ref{presystem2}, we have $|\{x,y,\a,\b\}|=4$. By item \ref{presystem1}, the sum of the four rows vanishes modulo two, that is every block contains an even number of the four elements. Conversely, if all four elements are distinct and $|\{x,y,\a,\b\}\cap B_j|\in\{0,2,4\}$ for every block $B_j$ of $\D$, then $m_x+m_y+m_\a+m_\b=0$ in $\F^b$ and thus, by item \ref{presystem1}, $\{x,y\}\sim_\D\{\a,\b\}$. 
\end{enumerate}
\hfill $\Box$

Due to item \ref{presystem2} of Lemma \ref{lemma_presystem}, if each equivalence class $\P_l(\D)$ contains the same number of elements (pairs), then $\P(\D)$ is a system of pairs. 
In Corollary \ref{cor_projective} below, we prove that $\P(\D)$ is a system of pairs if $\D$ is a projective plane. 
\\

While it is natural to suggest that if $\Z^+=\Z^+(\D)$ is an incidence matrix of a block design, then $\P(\D)$ is a system of pairs, this is not true as is shown in the next example.
\begin{example} 
\label{ex_counter}
Consider the following symmetric design $\D$  with parameters $(v=b=16, r=k=6, \lambda=2)$  The blocks of $\D$ are represented by columns: 
\begin{equation*}
\small
\begin{array}{cccccccccccccccc}
0 & 1 & 0 & 0 & 2 & 3 & 0 & 1 & 1 & 2 & 3 & 0& 0 & 1 & 2 & 3 \\
1 & 2 & 2 & 1 & 4 & 5 & 4 & 4 & 6 & 7 & 4 & 5& 5 & 6 & 7 & 4\\
2 & 3 & 3 & 3 & 5 & 6 & 6 & 5 & 8 & 9 & 8 & 8& 10 & 11 & 8 & 9 \\
4 & 5 & 6 & 7 & 6 & 7 & 7 & 7 & 9 & 10 & 10 & 9& 12 & 13 &12&12\\ 
9 & 10 & 11 & 8 & 8 & 9 & 10 & 11 & 10 & 11 & 11 & 11& 13 & 14 &14&13\\ 
14 & 15 & 12 & 13 & 13 & 14 & 15 & 12 & 12 & 13 & 14 & 15& 14 & 15 &15&15 
\end{array}
\end{equation*}
Due to the symmetry of $\D$, by Corollary \ref{cor_symmetric}, we obtain a design $D'$ from $\Z^+(\D)$. However, 
the equivalence relation of Definition \ref{def_equivalence} on 120 unordered pairs gives 54 classes of two different sizes: six classes with 4 pairs and forty-eight classes with 2 pairs.
\end{example}

On the other hand, in the next example, $\P(\D)$ is a system of pairs even though $\Z^+(\D)$ does not correspond to a block design. 
 Theorem \ref{thm_iff} below covers the case of this example. 
\begin{example} Consider the following simple design $\D$  with parameters $(v=8, b=14, r=7, k=4, \lambda=3)$ taken from \cite{handbook}, page 27, 1.21,  design 4. The blocks of $\D$ are represented by columns: 
\begin{equation*}
\small
\begin{array}{cccccccccccccc}
0 & 0 & 0 & 0 & 0 & 0 & 0 & 1 & 1 & 1 & 1 & 2& 2 & 4 \\
1 & 1 & 1 & 2 & 2 & 3 & 3 & 2 & 2 & 3 & 3 & 3& 3 & 5 \\
2 & 4 & 6 & 4 & 5 & 4 & 5 & 4 & 5 & 4 & 5 & 4& 6 & 6 \\
3 & 5 & 7 & 6 & 7 & 7 & 6 & 7 & 6 & 6 & 7 & 5& 7 & 7\end{array}.
\end{equation*}
There are two values of $\mu_{ij}$, namely $\mu_1=0$ and $\mu_2=2$. However, condition \eqref{quadratic} is not satisfied, so according to Theorem \ref{thm_main}, $\Z^+$ is not an incidence matrix of a block design. However, the equivalence relation on the pairs of elements of $V=\{1, \dots, 8\}$ induced by  $\Z^+$ as in Definition \eqref{def_equivalence}, gives rise to the following symmetric (but not simple) system of pairs: 
{\small
\begin{align*}
&\P_1(\D): \; \{0,1\}, \quad \{2,3\}, \quad \{4,5\}, \quad \{6,7\};
\\
&\P_2(\D): \; \{0,2\}, \quad \{1,3\}, \quad \{4,6\}, \quad \{5,7\};
\\
&\P_3(\D): \; \{0,3\}, \quad \{1,2\}, \quad \{4,7\}, \quad \{5,6\};
\\
&\P_4(\D): \; \{0,4\}, \quad \{1,5\}, \quad \{2,6\}, \quad \{3,7\};
\\
&\P_5(\D): \; \{0,5\}, \quad \{1,4\}, \quad \{2,7\}, \quad \{3,6\};
\\
&\P_6(\D): \; \{0,6\}, \quad \{1,7\}, \quad \{2,4\}, \quad \{3,5\};
\\
&\P_7(\D): \; \{0,7\}, \quad \{1,6\}, \quad \{2,5\}, \quad \{3,4\}.
\end{align*}
}
\end{example}
Examples \ref{example_Fano}-\ref{example_designpairs2} below give designs $\D$ for which matrices $\Z^+$ do correspond to block designs $\D'$ and $\P(D)$ are systems of pairs.

\begin{theorem}
\label{thm_symmetric}
For any design $\D$, the set of equivalence classes $\P(\D)$ is  symmetric. In particular, if $\P(\D)$ is a system of pairs, then it is a symmetric system of pairs.
\end{theorem}

{\it Proof.} Using notation $m_x$ from \eqref{mod2},  we need to prove that if for four distinct elements $x,y,\a,\b$,  we have $m_x+m_y=m_\a+m_\b$ in $\F^b$ then we also have $m_x+m_\a=m_y+m_\b$ and $m_x+m_\b=m_\a+m_y$. This is satisfied   in $\F^b$.
 
\hfill $\Box$

\begin{lemma}
\label{lemma_neclambda1}
Let $\D$ be a design with parameters $(v, b,r, k, \lambda=1)$. If there exist two pairs of elements of $V$ which are equivalent to each other in the sense of Definition \ref{def_equivalence}, then $r=3.$
\end{lemma}

{\it Proof.} Suppose that $\{x,y\}\sim_\D \{\a,\b\}$ and $\{x,y\}\ne \{\a,\b\}$. This means that for every $j=1, \dots, v$, we have that $|\chi_{B_j}(x) - \chi_{B_j}(y)|=1$ if and only if $|\chi_{B_j}(\a) - \chi_{B_j}(\b)|=1$. This implies that
for every $j$ such that $\Z_j(x,y)=1$ we have $\{x, y\}\not\subset B_j$, $\{\a, \b\}\not\subset B_j$ and one of the following four cases: 
\begin{equation}
\label{four-pairs}
1.\, \{x, \a\}\subset B_j, \qquad
2.\, \{x, \b\}\subset B_j, \qquad
3.\, \{y, \a\}\subset B_j, \qquad
4.\,  \{y, \b\}\subset B_j.
\end{equation}

Given that each pair appears exactly once in the blocks of $\D$,  if there exists a block $B_{j_*}$ such that  $\{x,y\}\subset B_{j_*}$ and $\{\a,\b\}\subset B_{j_*}$, then none of the four pairs \eqref{four-pairs} may appear in blocks different from $B_{j_*}$. Due to the equivalence of the pairs $\{x,y\}$, $\{\a,\b\}$, this implies that $\Z^+(x,y)=\Z^+(\a,\b)=0.$ However,  the vector $\Z^+(x,y)$ is nonzero for every $(x,y)$ since $k<v$ and thus $r>\lambda$. Therefore such a block $B_{j_*}$ does not exist.

Moreover, no block of $\d$ may contain exactly three elements of the set $\{x, y, \a, \b\}$ as in such a case the pairs $\{x,y\}$, $\{\a,\b\}$ would not be equivalent, see item \ref{presystem3} of Lemma \ref{lemma_presystem}. 

Thus the four pairs \eqref{four-pairs} appear in exactly four different blocks, one pair per block. We conclude that the number of nonzero components  of $\Z^+(x,y)$ is 4. On the other hand, due to Theorem \ref{thm_main}, this number is $k'=2(r-\lambda)$. We thus obtain that $2(r-1)=4$ or $r=3.$
\hfill $\Box$

\begin{corollary}
\label{cor_projective}
Let $\d$ be a finite projective plane. Then $\P(\d)$ is a system of pairs. However, the only finite projective plane that gives rise to a nontrivial system of pairs is the Fano plane.
\end{corollary}
{\it Proof.} Given the relations $bk=vr$ and $\lambda(v-1) = r(k-1)$,  in the case of a projective plane ($v=b$, $r=k$ and $\lambda=1$), Lemma \ref{lemma_neclambda1} implies that if there exist equivalent pairs, then  $k=r=3$ and $v=b=7$.  
There exists a unique design with these parameters, the Fano plane. Therefore, for all other projective planes $\D$, the conditions of Lemma \ref{lemma_neclambda1} are not satisfied, that is there are no two pairs of elements (of points of the plane) that are equivalent to each other in the sense of Definition \ref{def_equivalence}. Thus, the equivalence classes of pairs $\P_l(\D)$ contain exactly one pair each giving a trivial system of pairs. 
In Example \ref{example_Fano}, we show  that the Fano plane  does give rise to a nontrivial system of pairs. 
\hfill $\Box$

\begin{example}
\label{example_Fano}
Let $\D$ be the Fano plane, that is the design with parameters $(v=b=7, r=k=3, \lambda=1)$. The system of pairs $\P(\D)$ has parameters (v=7, n=7, t=3) and is as follows: 
{\small
\begin{align*}
&\P_1(\D): \; \{0,1\}, \quad \{2,5\}, \quad \{4,6\};\qquad &\P_5(\D): \; \{0,5\}, \quad \{1,2\}, \quad \{3,6\};
\\
&\P_2(\D): \; \{0,2\}, \quad \{1,5\}, \quad \{3,4\}; \qquad &\P_6(\D): \; \{0,6\}, \quad \{1,4\}, \quad \{3,5\};
\\
&\P_3(\D): \; \{0,3\}, \quad \{2,4\}, \quad \{5,6\}; \qquad &\P_7(\D): \; \{1,3\}, \quad \{2,6\}, \quad \{4,5\}.
\\
&\P_4(\D): \; \{0,4\}, \quad \{1,6\}, \quad \{2,3\};
\end{align*}
}
\end{example}

Let us now characterize designs $\D$ such that $\P(\D)$ is a system of pairs.
 Suppose $\P(\D)$ is a nontrivial system of pairs, then we can construct the design $\D_2(\P(\D))$ from the proof of item \ref{2} of Theorem \ref{thm_pairstoblocks}. The blocks of the design $\D_2(\P(\D))$ are precisely all the four-element sets having either 0, 2 or 4 elements in common with each block of $\D$.
 The next theorem shows that the collection of such four-element sets being a block design gives also a sufficient condition for $\P(\D)$ to be a system of pairs, independently of whether or not $\D'$ is a block design.

Denote by $\Q_4$ the following set  of four-element subsets of $V$:
\begin{equation}
\label{q4}
\Q_4 = \{ S\subseteq V \;:\; |S|=4 \;\; \text{ and }\;\; |S\cap B_j|\equiv 0 \,({\rm mod}\, 2)\;\; \text{ for all }\;\; 1\leqslant j\leqslant b \}. 
\end{equation}
\begin{theorem}
\label{thm_iff}
Let $v\neq 4$ and let $\D$ be a design with parameters $(v, b,r, k, \lambda)$. Then $\P(\D)$ is a system of pairs with parameters $(v, n, t)$ if and only if either $\Q_4 = \emptyset$ (then $\P(\D)$
is the trivial system, $t=1$), or $\Q_4$ is the block set of a design, necessarily with parameters $(v_4=v, k_4=4, \lambda_4=t-1>0)$.
In that case all subsystems of $\P(\D)$ have size $t= 1 + \lambda_4$.
\end{theorem}
{\it Proof.}
If $\P(\D)$ is a nontrivial system of pairs and
every class has size $t>1$, then by item \ref{presystem3} of Lemma \ref{lemma_presystem} every pair lies in exactly $t-1$ members of $Q_4$,
so $\lambda_4=t-1.$ (In this case, the construction of $\D_2(\P(\D))$ in the proof of item \ref{2} of Theorem \ref{thm_pairstoblocks}, implies that 
 $\Q_4$ is the set of blocks of the design $\D_2(\P(\D))$.)
For a trivial system of pairs, \ref{3} of Lemma \ref{lemma_presystem} implies that $\Q_4=\emptyset.$

Suppose now that $\Q_4$  forms a set of blocks of a design with parameters $(v_4=v, k_4=4, \lambda_4)$. Then for a pair $\{x,y\}$, we have $\lambda_4=|\{Q\in\Q_4\;:\; \{x,y\}\subset Q\}|.$ By item \ref{presystem3} of Lemma \ref{lemma_presystem}, the class of $\{x,y\}$ in $\P(\D)$ consists of the pair $\{x,y\}$ and all pairs $\{\a, \b\}\neq\{x,y\}$ such that
$\{x,y,\a,\b\}\in\Q_4.$ Thus the size of this class is $1+\lambda_4$, which is independent of $\{x,y\}$ 
given that $\Q_4$ is a block design. This shows that all equivalence classes of $\P(\D)$ are of equal size. Moreover, by item \ref{presystem2} of Lemma \ref{lemma_presystem},  no element of $V$ is contained in a  class in $\P(\D)$ more than once.

Condition $v\neq 4$ is included for the same reason as in item \ref{2} of Theorem \ref{thm_pairstoblocks}. 
\hfill $\Box$ 
\\

The next proposition describes a class of block designs $\D$ for which $\P(D)$ is a system of pairs. 
\begin{proposition}
\label{prop_auto}
Let $\D$ be a design with parameters $(v, b,r, k, \lambda)$. If there is a subgroup $G< {\rm Aut}(\D)$ of the automorphism group of $\D$ acting transitively on the set of unordered pairs $V^{(2)}$, then $\P(\D)$ is a system of pairs. 
\end{proposition}
{\it Proof.} For an automorphism $\phi\in {\rm Aut}(\D)$ let $\pi_\phi$ denote the induced permutation on the set of blocks: $B_{\pi_\phi(j)} = \phi(B_j)$. Then we have for the rows of the incidence matrix: $(m_x)_j=(m_{\phi(x)})_{\pi_\phi(j)}$ and thus $Z^+_{\pi_\phi(j)}(\phi(x), \phi(y)) = Z^+_j(x,y).$ This implies that $\{x,y\}\sim_\D\{\a, \b\}$ if and only if $\{\phi(x),\phi(y)\}\sim_\D\{\phi(\a), \phi(\b)\}.$ Therefore, $\phi$ permutes classes in $\P(\D)$. Moreover, denoting $\rho_\phi$ the induced permutation, $|P_l(\D)|=|P_{\rho_\phi(l)}(\D)|$.

Given that $\P(\D)$ is a partition of the set of unordered pairs $V^{(2)}$ and $G$ acts transitively on $V^{(2)}$, we conclude that $G$ acts transitively on the classes in $\P(\D)$. Indeed, for $\{x, y\}\in\P_s(\D)$ and $\{x', y'\}\in\P_q(\D)$, there is $\phi\in G$ such that $\{x',y'\}=\{\phi(x), \phi(y)\}$ and thus $q=\rho_\pi(s)$. This shows that all classes in  $\P(\D)$ are of the same size. 
\hfill $\Box$
\\

Let us conclude this section with some more examples of systems of pairs arising from block designs. 
\begin{example}
\label{example_designpairs2}
The symmetric design $\D$ with parameters $(v=b=16,  r=k=6, \lambda=2)$ from \cite{Hall}, (14.1.15) on p. 241, gives a design $D'$ with parameters $(v'=16, b'=120, r'=60, k'=8, \lambda'=28)$. The corresponding  equivalence relation for the pairs on the set $V=\{1, \dots, 16\}$ gives the following symmetric 
system of pairs:
\begingroup
\allowdisplaybreaks
\begin{eqnarray*}
\small
\begin{array}{ccccc}
\P_1(\D): &\!\!\!\!\{0, 1\},  &  \{3, 6\}, &  \{4, 7\},  &  \{10, 13\}; \\
\P_2(\D): &\!\!\!\!\{0, 2\},  &  \{3, 8\}, &  \{4, 9\},  &  \{10, 14\}; \\
\P_3(\D): &\!\!\!\!\{0, 3\}, &  \{1, 6\}, &  \{2, 8\},  &  \{5, 11\}; \\
\P_4(\D): &\!\!\!\!\{0, 4\}, &  \{1, 7\}, &  \{2, 9\},  &  \{5, 12\}; \\
\P_5(\D): &\!\!\!\!\{0, 5\}, &  \{3, 11\}, &  \{4, 12\},  &  \{10, 15\}; \\
\P_6(\D): &\!\!\!\!\{0, 6\}, &  \{1, 3\},  &  \{9, 15\}, &  \{12, 14\}; \\
\P_7(\D): &\!\!\!\!\{0, 7\}, &  \{1, 4\},  &  \{8, 15\}, &  \{11, 14\}; \\
\P_8(\D): &\!\!\!\!\{0, 8\}, &  \{2, 3\},  &  \{7, 15\}, &  \{12, 13\}; \\
\P_9(\D): &\!\!\!\!\{0, 9\},  &  \{2, 4\},  &  \{6, 15\}, &  \{11, 13\}; \\
\P_{10}(\D): &\!\!\!\!\{0, 10\}, &  \{1, 13\}, &  \{2, 14\},  &  \{5, 15\}; \\
\P_{11}(\D): &\!\!\!\!\{0, 11\}, &  \{3, 5\}, &  \{7, 14\}, &  \{9, 13\}; \\
\P_{12}(\D): &\!\!\!\!\{0, 12\},  &  \{4, 5\}, &  \{6, 14\}, &  \{8, 13\}; \\
\P_{13}(\D): &\!\!\!\!\{0, 13\}, &  \{1, 10\},  &  \{8, 12\}, &  \{9, 11\}; \\
\P_{14}(\D): &\!\!\!\!\{0, 14\},  &  \{2, 10\},  &  \{6, 12\}, &  \{7, 11\}; \\
\P_{15}(\D): &\!\!\!\!\{0, 15\}, &  \{5, 10\}, &  \{6, 9\}, &  \{7, 8\}; \\
\end{array}
\qquad
\begin{array}{ccccc}
\P_{16}(\D): &\!\!  \{2, 5\},  &  \{8, 11\}, &  \{9, 12\},  &  \{14, 15\}; \\
\P_{17}(\D):  &\!\!  \{1, 5\},  &  \{6, 11\}, &  \{7, 12\}, &  \{13, 15\}; \\
\P_{18}(\D):  &\!\!  \{4, 10\},  &  \{7, 13\}, &  \{9, 14\}, &  \{12, 15\}; \\
\P_{19}(\D):  &\!\!  \{3, 10\},  &  \{6, 13\}, &  \{8, 14\}, &  \{11, 15\}; \\
\P_{20}(\D):  &\!\!  \{1, 2\},  &  \{6, 8\}, &  \{7, 9\},  &  \{13, 14\}; \\
\P_{21}(\D): &\!\!  \{2, 11\}, &  \{4, 13\}, &  \{5, 8\}, &  \{7, 10\}; \\
\P_{22}(\D):  &\!\!  \{2, 12\}, &  \{3, 13\}, &  \{5, 9\}, &  \{6, 10\}; \\
\P_{23}(\D):  &\!\!  \{1, 11\},  &  \{4, 14\}, &  \{5, 6\},  &  \{9, 10\}; \\
\P_{24}(\D): &\!\!  \{1, 12\},  &  \{3, 14\}, &  \{5, 7\},  &  \{8, 10\}; \\
\P_{25}(\D):  &\!\!  \{3, 4\},  &  \{6, 7\}, &  \{8, 9\}, &  \{11, 12\}; \\
\P_{26}(\D):  &\!\!  \{1, 8\}, &  \{2, 6\},  &  \{4, 15\},  &  \{10, 12\}; \\
\P_{27}(\D):  &\!\!  \{1, 9\}, &  \{2, 7\}, &  \{3, 15\},   &  \{10, 11\}; \\
\P_{28}(\D):  &\!\!  \{2, 15\}, &  \{3, 7\}, &  \{4, 6\}, &  \{5, 14\}; \\
\P_{29}(\D):  &\!\!  \{1, 15\}, &  \{3, 9\}, &  \{4, 8\}, &  \{5, 13\}; \\
\P_{30}(\D):  &\!\!  \{1, 14\}, &  \{2, 13\}, &  \{3, 12\}, &  \{4, 11\}. \\
\end{array}
\end{eqnarray*}
\endgroup
\end{example}

If $\D$ is such that $\Z^+$ is an incidence matrix for a block design $\D'$ and  $\P(\D)$ is a nontrivial system of pairs, then $\D'$ contains repeated blocks. Moreover, every set of identical blocks of $\D'$ contains the same number of blocks, let us denote it by $t>1$. In this case, identifying all identical blocks of $\D'$, that is only keeping one block in each set of identical blocks, we obtain another design $\tilde\D'$. If $\D'$ has parameters $(v', b', r', k', \lambda')$ then $\tilde\D'$ has parameters $(\tilde v'=v', \tilde b'=\frac{b'}{t}, \tilde r'=\frac{r'}{t}, \tilde k'=k', \tilde \lambda'=\frac{\lambda'}{t})$. Let us consider this situation in Example \ref{example_designpairs}.

\begin{example}
\label{example_designpairs}
Consider the simple design  $\D$  with parameters  $(v=6,\,  b=10,\,  r=5,\,  k=3, \, \lambda=2)$ from \cite{Hall},  Table 1, p. 406 , design number 4; its blocks are  the columns of the following array: 
\begin{equation*}
\small
\begin{array}{cccccccccc}
0 & 0 & 0 & 0 & 0 & 1 & 1 & 1 & 2 & 2   \\
1 & 1 & 2 & 3 & 3 & 2 & 3 & 4 & 3 & 4 \\
2 & 5 & 4 & 4 & 5 & 3 & 4 & 5 & 5 & 5  \\
\end{array}.
\end{equation*}
 The intersection numbers $\mu_{ij}$ take two values, 1 and 2, and the conditions of Theorem \ref{thm_main} are satisfied. The resulting design $\D'$ has parameters  $(v'=10,\,  b'=15,\,  r'=9,\,  k'=6, \, \lambda'=5)$.
The equivalence relation on the columns of $\Z^+=\Z^+(\D)$ gives rise to a trivial system of pairs on the set $V=\{1, \dots, 6\}$ (the design $\D'$ is  simple). Consider now the constructed design $\D'.$ Its intersection numbers $\mu'_{ij}$ take
 again only two values: $\mu'_{ij}\in\{3,4\}$ satisfying the second condition in \eqref{quadratic}. Therefore, by Theorem \ref{thm_main}, the matrix $\Z^+(\D')$ built from  $\D'$ is an incidence matrix of a block design $\D'':=(\D')'$ with parameters $v''=15,\, b''=45,\,  r''=24,\,  k''=8$ and $\lambda''=12.$ The equivalence relation from Definition \ref{def_equivalence} with respect to $\D'$ on  the pairs of the set $V'=\{1, \dots, 10\}$ gives rise to the following symmetric system of pairs with parameters $(v=10, n=15, t=3)$: 
\begin{eqnarray*}
\small
\left.
\begin{array}{cccc}
\P_1(\D): &\!\!\!\!\{1,8 \}, &\!\! \{2,6 \}, &\!\! \{3,4 \};  \\
\P_2(\D): &\!\!\!\!\{1,9 \}, &\!\! \{3,5 \}, &\!\! \{6,7 \}; \\
\P_3(\D): &\!\!\!\!\{ 2,7 \}, &\!\! \{ 4,5 \}, & \!\!\{8 ,9 \};  \\
\P_4(\D): &\!\!\!\!\{0 ,1 \}, &\!\! \{2 , 5\}, &\!\! \{ 4, 7\}; \\
\P_5(\D): &\!\!\!\!\{ 0, 8\}, &\!\! \{ 3,7 \}, & \!\!\{5 ,6 \};  \\
\end{array}
\right. 
\quad
\left.
\begin{array}{cccc}
\P_6(\D): &\!\!\!\!\{ 0, 4\}, &\!\! \{1, 7\}, &\!\! \{6 ,9 \}; \\
\P_7(\D): &\!\!\!\!\{ 0,3 \}, &\!\! \{2 , 9\}, &\!\! \{7 ,8 \}; \\
\P_8(\D): &\!\!\!\!\{0 ,9 \}, &\!\! \{2 ,3 \}, &\!\! \{ 4, 6\};  \\ 
\P_9(\D): &\!\!\!\!\{ 2, 4\}, &\!\! \{ 3,6 \}, &\!\! \{5 ,7 \}; \\
\P_{10}(\D): &\!\!\!\!\{0  ,6 \}, &\!\! \{4 , 9 \}, & \!\!\{5 , 8\};  
\end{array}
\right. 
\quad
\left.
\begin{array}{cccc}
\P_{11}(\D): &\!\!\!\!\{ 0, 2\}, &\!\! \{1 ,5 \}, &\!\! \{ 3, 9\}; \\
\P_{12}(\D): &\!\!\!\!\{0 , 7\}, &\!\! \{1 , 4\}, &\!\! \{ 3,8 \};  \\
\P_{13}(\D): &\!\!\!\!\{1 ,3 \}, & \!\!\{4 , 8\}, &\!\! \{ 5,9 \}; \\
\P_{14}(\D): &\!\!\!\!\{ 1, 6\}, &\!\! \{2 , 8\}, & \!\!\{7 ,9 \};  \\
\P_{15}(\D): &\!\!\!\!\{0 , 5\}, &\!\! \{ 1, 2\}, & \!\!\{6 ,8 \}.
\end{array}
\right.
\end{eqnarray*}
Identifying identical blocks of $\D''$, we obtain a design $\tilde \D''$ with parameters $(\tilde v''=15, \tilde b''=15, \tilde r''=8, \tilde k''=8, \tilde \lambda''=4).$ It turns out that, starting from $\tilde \D''$, we can iterate the described process an infinite number of times while identifying identical blocks after every iteration. In this way, we obtain a finite number of designs, each with the same parameters as $\tilde \D''$.

\end{example}

\section{Relationship of friendship between designs $\D$ and $\D_2(\P(\D))$}
\label{sect_friends}

Let us now remind a relationship between block designs built on the same set introduced in \cite{friends}. To this end, for a block design $\D$ built on a set $V$ and some subset $S$ of $V$, define $\rho_j\in\mathbb N\cup\{0\}$ as the number of blocks of $\D$ whose intersections with $S$ contain exactly $j$ elements. The numbers $\rho_j$ satisfy the following relations for $(v,b,r,k,\lambda)$ being the parameters of $\D$ and $s$ being the number of elements in the set $S$  (see Lemma 4 from \cite{friends}): 
\begin{equation}
\label{eqs}
\sum_{j=0}^k \rho_j=b,\qquad
\sum_{j=0}^k \rho_j{j}=rs, \qquad 
\sum_{j=0}^k \rho_j{j^2}=s(\lambda s-\lambda+r).
\end{equation} 
The second relation follows by considering the number of elements in all the intersections of $S$ and the blocks of $\D$. To obtain the third relation, one focuses on the total number of pairs of elements in these intersections. 

Due to the first equation in \eqref{eqs}, the set of numbers $\{\rho_0, \rho_1, . . . , \rho_k\}$ is a partition of the number of blocks of  $\D$. Let us denote this ordered partition by $\rho(\D,S)=(\rho_0, \rho_1, . . . , \rho_k).$ 
 \begin{definition}\cite{friends}
Two block designs $\D^{(1)}$ and $\D^{(2)}$ built on the same set $V$ are called friends if the partitions $\rho(\D^{(1)},B^{(2)}_i)$ and $\rho(\D^{(2)},B^{(1)}_j)$ do not depend on $i$ and $j$, respectively. Here $B^{(j)}_i$ are the blocks of the design $\D^{(j)}$.
\end{definition}
\begin{proposition}
\label{prop_friends}
Let $\D$ be a block design built on a set $V$ such that $\P(\D)$ is a system of pairs. Let $\D_2(\P(\D))$ be the block design from the proof of  item \ref{2} of Theorem \ref{thm_pairstoblocks} built using $\P(\D)$. 
Then $\D$ and $\D_2(\P(\D))$ are friends. 
\end{proposition}
{\it Proof.} Let $(v,b,r,k, \lambda)$ be the parameters of $\D$. Denote by  $Q_j$ the blocks of  $\D_2(\P(\D))$; the blocks of $\D$ are denoted by $B_i$ as before. Let us prove that $\rho(\D, Q_j)$ does not depend on $j$. By construction, $Q_j$ is a union of two pairs $\{x,y\}$ and $\{\a,\b\}$ of elements of $V$ such that 
$\{x,y\} \sim_\D \{\a,\b\}$. By item \ref{presystem3} of Lemma \ref{lemma_presystem}, $\rho(\D, Q_j)=(\rho_0, 0, \rho_2, 0, \rho_4).$ Thus, the three possibly nonzero values $\rho_0, \rho_2, \rho_4$ can be found from equations \eqref{eqs} with $s=4$: 
\begin{equation}
\label{024}
\rho_0=b-\frac{1}{2}(5r-3\lambda), \qquad \rho_2=3(r-\lambda), \qquad \rho_4= \frac{1}{2}(3\lambda-r).
\end{equation} 
This proves that the partitions $\rho(\D, Q_j)$ are the same for all $j$.  

Let us now consider ordered partitions $\rho(\D_2(\P(\D)), B_i)$ of the number of blocks of $\D_2(\P(\D))$, that is of the number $\frac{n}{3}{t\choose 2}$, where $t$ is the number of pairs in the subsystems of $\P(\D)$ and $n$ is the number of subsystems, see item \ref{2} of Theorem \ref{thm_pairstoblocks}. For the same reasons as above, these partitions are of the form $\rho(\D_2(\P(\D)), B_i)=(\hat\rho_0, 0, \hat\rho_2, 0, \hat\rho_4)$ and thus the integers $\hat\rho_0, \hat\rho_2, \hat\rho_4$ may be determined from equations \eqref{eqs} with $b=b_2, r=r_2, k=k_2, \lambda=\lambda_2$ being the parameters of $\D_2(\P(\D))$ given by \eqref{D2param} and $s=|B_i|=k$. Thus the partitions $\rho(\D_2(\P(\D)), B_i)$ are independent of $i.$
\hfill $\Box$

\begin{corollary}
\label{cor_parity}
Let $\D$ be a block design with parameters $(v, b, r, k, \lambda)$ and  $\P(\D)=\{\P_1(\D), \dots, \P_n(\D)\}$ be the set of equivalence classes on the set $V^{(2)}$ induced by $\D$ as in Definition \ref{def_equivalence}. If one of the equivalence classes $\P_j(\D)$ contains more than one element (that is more than one pair of elements of $V$), then $r$ and $\lambda$ are of the same parity. 
In particular, if $\P(\D)$ is a nontrivial system of pairs, then $r$ and $\lambda$ are of the same parity. If $r$ and $\lambda$ are of different parities, then $\P(\D)$ is a trivial system of pairs.
\end{corollary}
{\it Proof.} Suppose $\P_j(\D)$ contains two pairs $\{\a,\b\}$ and $\{x,y\}$. By item \ref{presystem2} of Lemma \ref{lemma_presystem}, these two pairs do not have common elements. The intersection of the set $S=\{\a,\b,x,y\}$ with every block of $\D$ contains either 0, 2 or 4 elements and the corresponding numbers $\rho_0, \rho_2, \rho_4$ are given by \eqref{024}. The numbers \eqref{024} are integers if and only if $r$ and $\lambda$ are of the same parity. 
\hfill $\Box$

\vskip 0.5cm
{\bf Acknowledgments.} VS gratefully acknowledges
support from the Natural Sciences and Engineering Research Council of Canada (discovery grant) and the University of Sherbrooke.


\end{document}